\documentclass[12pt]{amsart}
\usepackage{amsmath}
\usepackage{amsfonts}
\usepackage{amsthm}
\usepackage{amssymb}
\usepackage{amscd}
\usepackage[all]{xy}
\usepackage{enumerate}
\usepackage{hyperref}
\usepackage{comment}
\usepackage{paralist}
\usepackage{tikz}
\usetikzlibrary{cd}
\usepackage{mathtools}
\usepackage[latin1]{inputenc}
\usepackage{mathrsfs}
\usepackage{array}


\subjclass[2010]{}

\theoremstyle{plain}
\newtheorem{thm}{Theorem}[section]

\newtheorem{prop}[thm]{Proposition}

\newtheorem{cor}[thm]{Corollary}
\newtheorem{rem}[thm]{Remark}
\newtheorem{lem}[thm]{Lemma}

\theoremstyle{definition}

\newcommand{\sL}{\mathcal{L}}

\newcommand{\mG}{\mathbb{G}}

\newcommand{\mP}{\mathbb{P}}

\usepackage{color}
\def\leq{\leqslant}
\numberwithin{equation}{section}

\newcommand{\beba}  {\begin{equation}\begin{array}{rcl}}

\newcommand{\eaee}  {\end{array}\end{equation}}

\makeatletter
\def\l@section{\@tocline{1}{0pt}{1pc}{}{}}
\def\l@subsection{\@tocline{2}{0pt}{1pc}{4.6em}{}}
\def\l@subsection{\@tocline{3}{0pt}{1pc}{7.6em}{}}
\renewcommand{\tocsection}[3]{%
  \indentlabel{\@ifnotempty{#2}{\makebox[2.3em][l]{%
    \ignorespaces#1 #2.\hfill}}}#3}
\renewcommand{\tocsubsection}[3]{%
  \indentlabel{\@ifnotempty{#2}{\hspace*{2.3em}\makebox[2.3em][l]{%
    \ignorespaces#1 #2.\hfill}}}#3}
\renewcommand{\tocsubsection}[3]{%
  \indentlabel{\@ifnotempty{#2}{\hspace*{4.6em}\makebox[3em][l]{%
    \ignorespaces#1 #2.\hfill}}}#3}
\makeatother

\title[On the degree of a modular map]{On the degree of a modular map}

\author{Ciro Ciliberto}
\address{
Ciro Ciliberto\\Dipartimento di Matematica\\
Universit\`a di Roma ``Tor Vergata''\\
Via della Ricerca  Scientifica, 00177 Roma\\ Italia
\texttt{cilibert@axp.mat.uniroma2.it   }}

\author{Sandro Verra}
\address{Sandro Verra\\Departimento di Matematica e  Fisica \\
Universit\`a Roma 3\\
Largo San Leonardo Murialdo,  00146 Roma \\ Italia
\texttt{sandro.verra@gmail.com}}

\begin{document}

\markboth{}{}
	\label{cast2}
\maketitle
%\tableofcontents

\begin{abstract} {Let $X$ be a general cubic hypersurface in $\mathbb P^4$. If $x\in X$ is a general point there are exactly six distinct lines in $X$ passing through $x$, that lie on the rank 3 quadric cone with vertex $x$ of lines that have intersection multiplicity at least 3 with $X$ in $x$. So there is a natural rational map $X\dasharrow \mathcal M_2$. In this paper we compute its degree to be 2074320. } \end{abstract}

\tableofcontents

\section{Introduction} Let $X$ be a general cubic hypersurface in the complex projective space $\mathbb P^4$. If $x\in X$ is a general point there are exactly six distinct lines in $X$ passing through $x$, that lie on the rank 3 quadric cone with vertex $x$ of lines that have intersection multiplicity at least 3 with $X$ in $x$. Hence to $x\in X$ one can associate a set of six distinct points lying on an irreducible conic, and ultimately the genus 2 curve that is the double cover of the conic branched at those six points.  So there is a natural rational map $\phi: X\dasharrow \mathcal M_2$ that turns out to be dominant. 
The aim of this paper is to compute the degree of the map $\phi$. This is given in Theorem \ref {thm:main} below, and it turns out to be 2074320.

The idea of the proof is as follows. We consider (see \S \ref {ssec: deg}) a general degeneration of a general cubic hypersurface $X$ in $\mathbb P^4$ to a very special cubic hypersurface $X_0$, namely a \emph{Segre cubic}, i.e., a cubic hypersurface with 10 ordinary double points and no other singularity (of which we recall the properties necessary for us in Section \ref {sec:segprim}) . At a first glance, the map $\phi$ seems to degenerate to the analogous map $\phi_0$ for $X_0$, so that one would be led to assert that the degree of $\phi$ equals the degree of $\phi_0$,  that turns out to be $6!$ as we recall in Section \ref {sec:phi}. However this is not the case. Indeed a closer analysis shows that there are further contributions to the degree of $\phi$ coming from the 10 nodes of $X_0$ (see \S \ref {ssec:var}). These contributions are computed in \S \ref {ssec:var} with the help of the computation of the degree of another modular map considered in Section \ref {sec:anot}.

\medskip
\noindent {\bf Aknowledgement}:  The authors are members of GNSAGA of the Istituto Nazionale di Alta Matematica ``F. Severi''. \color{black}
\medskip

\section{The Segre primal} \label{sec:segprim}

The \emph{Segre primal} is a  cubic hypersurface in $\mP^4$, introduced by  C. Segre in \cite {Seg} and very much studied in the literature. For a recent reference, see Dolgachev's book \cite {Dol}. In this section we recall some basic facts about the Segre primal, that will be useful in the sequel. 

\subsection{Definition}  Consider the 5--dimensional projective space $\mP^5$ with homogeneous coordinates $[x_0,\ldots, x_5]$. We will identify the 4--dimensional projective space $\mP^4$ with the hyperplane of $\mP^5$ with equation
$$
\sum_{i=0}^5 x_i=0.
$$ 
The \emph{Segre primal} $X_0$ is the cubic hypersurface of $\mP^4$ with equations
$$
\sum_{i=0}^5 x^3_i=0, \quad \sum_{i=0}^5 x_i=0. 
$$

The symmetric group $\mathfrak S_6$ acts via projectivities on $\mP^5$ by permuting the homogeneous coordinates. This action fixes the Segre primal, which is therefore acted on by $\mathfrak S_6$. One has ${\rm Aut}(X_0)=\mathfrak S_6$ (see \cite {CTZ}).

\subsection{Double points and planes} The singular locus of the Segre primal consists of exactly 10 ordinary double points, precisely the point $[1,1,1,-1,-1,-1]$ and the points obtained from this by permuting the coordinates (see \cite [p. 472] {Dol}).  The group $\mathfrak S_6$ acts transitively on the set of double points of the Segre primal. 

It is well known that a cubic hypersurface in $\mP^4$ whose singular locus consists of exactly 10 ordinary double points is projectively equivalent to the Segre primal (see \cite [Thm. 9.4.14]{Dol}). 

The Segre primal contains also 15 planes, that have the equations
$$
x_i+x_j=x_k+x_l=x_m+x_n=0
$$
with $\{i,j,k,l,m,n\}=\{0,1,2,3,4,5\}$ (see \cite [p. 472] {Dol}).  Again, $\mathfrak S_6$ acts transitively on the set of planes of the Segre primal. 

The configuration of double point and planes of the Segre primal is interesting. It is well known that it is a configuration of type $(6,4)$, namely, every plane contains 4 double points and through every double point pass exactly 6 planes (see, e.g., \cite [Prop. 2.2] {Dol2}).

\subsection{Rationality}\label{ssec:rat}  The Segre primal is rational. A birational map  $\sigma: X_0\dasharrow \mP^3$ is given by projecting from a double point. The inverse map $\sigma^{-1}: \mP^3\dasharrow X_0$, is given by a 4--dimensional linear system of cubics. However we will not dwell on this here, because we are interested in a simpler birational map $\mP^3\dasharrow X_0$ defined by a 4--dimensional  linear system of quadrics. 

Let $p_1,\ldots,p_5$ be independent points in $\mP^3$, and consider the 4--dimensional linear system $\sL$ of quadrics containing $p_1,\ldots, p_5$. The image of the rational map $\varphi_\sL: \mP^3\dasharrow \mP^4$ determined by $\sL$, whose indeterminacy locus is  $\{p_1,\ldots,p_5\}$, is exactly a Segre  primal   (see \cite [Prop. 9.4.15]{Dol}). 

The geometry of the Segre primal can be read off by the map $\varphi_\sL$. For example:\\
\begin{inparaenum}
\item [$\bullet$] the 10 lines joining pairs of the points $p_1,\ldots, p_5$ are contracted by $\varphi_\sL$ to the 10 double points of the Segre primal;\\
\item [$\bullet$] the 15 planes of the Segre primal are the images of the points $p_1,\ldots, p_5$, that are blown up by $\varphi_\sL$, and of the 10 planes generated by all triples of points among $p_1,\ldots, p_5$.
\end{inparaenum}

We notice that the only subvarieties of $\mathbb P^3$ contracted to points via $\varphi_\sL$, are the 10 lines joining pairs of points among $p_1,\ldots, p_5$. So on the open subset complement of these 10 lines, the map $\varphi_\sL$ determines an isomorphism.

\subsection {Lines} \label{ssec:lines} For this section, see \cite [Ex. 10.15]{Dol}.

As a general cubic threefold, also the Segre primal $X_0$ has a 2--dimensional \emph{Fano variety} $F(X_0)$ of lines, that, considered as a surface in the Grassmannian $\mG(1,4)$ of lines in $\mathbb P^4$  has degree 45, as it happens for a general cubic threefold. First of all $F(X_0)$ contains the 15 planes. Now, given a general point $p\in X_0$, we expect that there are exactly 6 lines in $X_0$ passing through $p$. Indeed, if we introduce affine coordinates in $\mP^4$ with the origin at $p$, we see that $X_0$ is given by an affine equation of the form
\begin{equation}\label{eq:eq}
F_1(x_1,\ldots, x_4)+F_2(x_1,\ldots, x_4)+F_3(x_1,\ldots, x_4)=0
\end{equation}
where $F_i(x_1,\ldots, x_4)$ is a homogeneous polynomial of degree $i$ in $x_1,\ldots, x_4$, with $1\leq i\leq 3$. Then the set of lines on $X_0$ containing $p$ is defined by 
\begin{equation}\label{eq:eq2}
F_1(x_1,\ldots, x_4)=F_2(x_1,\ldots, x_4)=F_3(x_1,\ldots, x_4)=0,
\end{equation}
equations which are expected to define six distinct lines through $p$. To see that this is exactly the case, we consider again the birational map  $\varphi_\sL: \mP^3\dasharrow X_0$. Let $x$ be a general point of $\mP^3$, that, via the map $\varphi_\sL$ is sent to a general point $p$ of $X_0$. The 6 lines on $X_0$ passing through $p$ are the image via $\varphi_\sL$ of the 5 lines joining $x$ to $p_1,\ldots, p_5$ and of the unique rational normal cubic curve containing $x$ and $p_1,\ldots, p_5$. So $F(X_0)$ consists also of the images via $\varphi_\sL$ of:\\
\begin{inparaenum}
\item [(i)] the 5 stars of lines passing through each of the points $p_1,\ldots, p_5$;\\
\item [(ii)] the set of rational normal cubic curves  passing through the points $p_1,\ldots, p_5$.
\end{inparaenum}

So, besides the 15 planes, $F(X_0)$ contains  6 more surfaces, i.e., the images of the sets in (i) and (ii) above.  It is well known that they are all Del Pezzo surfaces of degree 5 and the group $\mathfrak S_6$ acts transitively on the set of these six surfaces, with stabilizer subgroup isomorphic to the group of automorphisms of the Del Pezzo surfaces (see \cite [\S 4]{Dol2}).

It is useful to record here, for further purposes, a few results about lines in $X_0$. 

\begin{lem}\label{lem:fin} Given any point $p\in X_0$ that does not lie on a plane on $X_0$, there are exactly six distinct lines in $X_0$ through $p$, that lie on the rank 3 quadric cone with vertex at $p$ of the lines having with $X_0$ at $p$ intersection multiplicity at least 3. 
\end{lem}

\begin{proof} Consider again the map $\varphi_\sL: \mP^3\dasharrow X_0$. Let $p'\in \mathbb P^3$ be the point such that $\phi_\mathcal L(p')=p$. By the assumption, $p'$ does not lie on any plane spanned by three of the points $p_1,\ldots, p_5$. Then the five lines joining $p'$ with each of the points $p_1,\ldots, p_5$ are distinct and mapped by $\varphi_\sL$ to five distinct lines in $X_0$ through $p$. Consider the rank 3 quadric cone $Z$ with vertex at $p'$ containing $p_1,\ldots, p_5$, that contains the aforementioned five lines. On this cone there is a unique rational normal cubic curve containing $p_1,\ldots, p_5$ and $p'$, and it is clear that it is irreducible. This rational normal cubic is mapped via $\varphi_\sL$ to a different line passing through $p$, and there is no other line in $X_0$ through $p$. 

To prove the last assertion, we notice that the cone $Z$ is mapped via $\varphi_\sL$ to the tangent hyperplane section $\mathfrak Z$ to $X_0$ at $p$, which is a cubic surface that, locally at $p$, is analytically isomorphic to $Z$ at $p'$, so it has an ordinary double point at $p$. Now the quadric cone with vertex at $p$ of the lines having with $X_0$ at $p$ intersection multiplicity at least 3 coincides with the analogous quadric cone for $\mathfrak Z$. Since $\mathfrak Z$ has an ordinary double point at $p$, the same is true for this quadric cone, proving the assertion.\end{proof} 

\begin{lem}\label{lem:pl} Let $p\in X_0$ be a point that lies on a plane $\pi$ in $X_0$ but is different from a double point. Then the set of lines in $X_0$ containing $p$ consists of the pencil of lines through $p$ in $\pi$ plus two more distinct lines not lying on $\pi$.
\end{lem}

\begin{proof} Given the transitive action of $\mathfrak S_6={\rm Aut}(X_0)$ on the planes in $X_0$, we can consider just one of these plane and prove the assertion for it. Look again the map $\varphi_\sL: \mP^3\dasharrow X_0$ and suppose that the plane $\pi$ is the image via $\varphi_\mathcal L$ of the plane $\pi'=\langle p_1,p_2, p_3\rangle$. The four double points on $\pi$ are the images via  $\varphi_\mathcal L$ of the three lines pairwise joining the points $p_1,p_2, p_3$, and of the intersection point $z=\langle p_4,p_5\rangle \cap \pi'$. So we assume that $p=\varphi_\mathcal L(p')$, with $p'\in \pi'$ off the three lines pairwise joining   $p_1,p_2, p_3$ and different from $z$.  Then the lines in $X_0$ passing through $p$  and not lying on the plane $\pi$ are the images via $\varphi_\mathcal L$ of the two distinct lines joining $p'$ with $p_4$ and $p_5$. The assertion follows. \end{proof}

\begin{rem}\label{rem:lim}{\rm Given a point $p\in X_0$ in a plane $\pi$ in $X_0$ distinct from the four double points on $\pi$, there are in general six marked lines of $X_0$  through $p$, namely the four lines joining $p$ with the four double points on $\pi$ plus the two lines passing through $p$ and not lying on $\pi$. At most two of these lines can coincide to one (with multiplicity 2), in case $p$ is collinear with two of the double points. By taking once more into account the map $\varphi_\sL$, it follows that these six lines are the limit of the lines in $X_0$ through a general point $x\in X_0$ when $x$ tends to $p$.}
\end{rem}

\begin{lem}\label{lem:lines} Let $p$ be a double point of the Segre primal $X_0$. Then the union of lines in $X_0$ passing through $p$ coincides with the union of the six planes in $X_0$ passing through $p$. These six planes are a complete intersection  of type $(2,3)$, they lie on the rank 4 tangent quadric cone  to $X_0$ at $p$, and three of them belong to one ruling of the cone, the other three to the other ruling. 
\end{lem}

\begin{proof} We introduce affine coordinates in $\mP^4$ with the origin at $p$. Then $X_0$  has  equation of the form
$$
F_2(x_1,\ldots, x_4)+F_3(x_1,\ldots, x_4)=0
$$
where $F_i(x_1,\ldots, x_4)$ is a homogeneous polynomial of degree $i$ in $x_1,\ldots, x_4$, with $2\leq i\leq 3$. The set of lines on $X_0$ containing $p$ is defined by 
$$
F_2(x_1,\ldots, x_4)=F_3(x_1,\ldots, x_4)=0,
$$
that is a complete intersection surface of type $(2,3)$. On the other hand the set of lines on $X_0$ containing $p$ contains the six planes in $X_0$ passing through $p$. Then the assertion  follows.
\end{proof}

%%%
\begin{comment}
It is useful to record here the following easy fact:

\begin{lem}\label{lem:inf} Let $p\in X_0$ be a point such that there are infinitely many lines of $X_0$ passing through $p$. Then $p$ sits in a plane of $X_0$.
\end{lem}

\begin{proof} 
 \end{proof}
\end{comment}
%%%???

\subsection{The Castelnuovo--Richmond quartic}\label{ssec:CR}  The dual $X_0^*$ of the Segre primal $X_0$ is a quartic hypersurface, called the \emph{Castelnuovo--Richmond quartic} (see \cite [p. 478]{Dol}), or sometimes the \emph{Igusa quartic} (see \cite [p. 445]{Dol}). The duality map
$$
\tau: X_0\dasharrow X_0^*
$$  
is birational, it is not defined only at the 10 double points of $X_0$ and maps the 15 planes of $X_0$ to 15 lines of $X_0^*$ that are loci of double points for $X_0^*$ (see \cite [p. 479]{Dol}). The Castelnuovo--Richmond quartic has an important modular property (see \cite [p. 545] {Dol}) that we will now recall.

Let $q$ be a general point of the Castelnuovo--Richmond quartic $X^*_0$. The tangent hyperplane section $H_q$ of $X^*_0$ at $q$ is a quartic surface with 16 nodes, one at $q$ and the other 15 at the intersections of the tangent hyperplane with the 15 double lines of $X^*_0$. Hence $H_q$  is a \emph{Kummer surface}, i.e., it is the quotient of the Jacobian $J(C_q)$ of a curve $C_q$ of genus $2$ by the $\pm 1$ involution. By projecting $H_q$ from the double point $q$, we obtain a double plane branched along the union of six distinct lines $\ell_1,\ldots, \ell_6$ tangent to a smooth conic $\Gamma$ (see \cite [Thm. 10.3.16]{Dol}). The double cover of $\Gamma$ branched at the six contact points of $\Gamma$ with $\ell_1,\ldots, \ell_6$ is just the curve $C_q$ and to specify the above set up is equivalent to give an ordering of the six Weierstrass points of $C_q$, or, what is the same, to give a choice of a symplectic basis in the group of 2--torsion points of $J(C_q)$, that is endowed with the \emph{Weil pairing} (see \cite [p. 188]{Dol}). In conclusion we have a rational map
$$
\psi: X^*_0\dasharrow \mathcal A_2(2)
$$
where $\mathcal A_2(2)$ is the moduli space of principally polarised abelian surfaces endowed with the choice of a symplectic basis of the group of 2--torsion points, or, what is the same, with a  level 2 structure. It is a fundamental result by Igusa (see \cite {Ig}, cfr. also \cite {G}) that
$\psi$ is birational.

\section{A modular map}\label{sec:phi}

If $p$ is any point in $X_0$ that does not lie on any plane in $X_0$,  there are six distinct lines on $X_0$ passing through $p$ that lie on the rank 3 quadric cone with vertex at $p$ of the lines having with $X_0$ at $p$ intersection multiplicity at least 3. So we can associate to $p$ the datum of a smooth conic plus six distinct points on it, and therefore we can consider the genus 2 double cover of this conic branched exactly at the six points. Accordingly we have a rational map
$$
\phi_0: X_0\dasharrow  \overline {{\mathcal M}_2}
$$
where, as usual, ${\mathcal M}_2$ is the moduli space of smooth, irreducible curves of genus 2. 

By Lemmata \ref {lem:fin}, \ref {lem:pl} and \ref {lem:lines}, and by Remark \ref {rem:lim} we have the following facts:\\
\begin{inparaenum}
\item [$\bullet$] the indeterminacy locus of the map $\phi_0$ is just the union of the 10 double points of $X_0$;\\
\item [$\bullet$] a point in $X_0$ is mapped to a point in $\mathcal M_2$ if and only if it does not lie on a plane contained in $X_0$;\\
\item [$\bullet$] the planes in $X_0$ are mapped via $\phi_0$ to the boundary component $\Delta_{0}$ of $\mathcal M_2$.
\end{inparaenum}

\begin{prop}\label{prop:deg} The map $\phi_0$ is generically finite of degree $6!$.
\end{prop}

\begin{proof} We have the composition $\psi_0:X_0\dasharrow  \mathcal A_2(2)$ of the two birational maps $\tau$ and $\psi$ introduced in Section \ref {ssec:CR}, so that $\psi_0$ is also birational. Then we have the rational map $\phi_0\circ \psi_0^{-1}:\mathcal A_2(2) \dasharrow \overline {\mathcal M}_2$, that is clearly of degree $6!$. The assertion follows. \end{proof}

\begin{rem}\label{rem:dol} {\rm Proposition \ref {prop:deg} says that the general fibre of $\phi_0$ is an orbit via the $\mathfrak S_6$ action.

According to \cite [\S 2]{Dol2}, Proposition \ref {prop:deg} has been first discovered by P. Joubert \cite {Jou} and then proved by A. Coble \cite {Co}.} \end{rem}

\section{Another modular map}\label{sec:anot}

\subsection{The definition} For further purposes, we have to consider the following situation. Let $\mathbb Q\cong \mathbb P^1\times \mathbb P^1$ be a smooth quadric in $\mathbb P^3$. We will consider a curve $C$ of degree 6 and arithmetic genus 4 on $\mathbb Q$ consisting of the union of three distinct lines of a ruling plus three distinct lines of the other ruling. 

Given a general plane $\pi$ in $\mathbb P^3$, it intersects $\mathbb Q$ along a smooth conic  $\Gamma$ and $C$ at 6 distinct points lying on $\Gamma$. So, as in Section \ref {sec:phi}, to $\pi$ it remains associated a point in $\mathcal M_2$ and therefore we have a rational map
$$
\varphi: (\mathbb P^3)^\vee \dasharrow \overline {\mathcal M_{2}}.
$$
We want to prove that this map is dominant and we want to compute its degree.

Notice that we can order the six components of $C$  as $(\ell_1,\ldots, \ell_6)$, with $\ell_1, \ell_2, \ell_3$ lying in one of the two rulings of $\mathbb Q$ and $\ell_4, \ell_5, \ell_6$ lying in 
the other ruling. Then also the 6 intersection points of $\Gamma$ and $C$ are naturally ordered, so that we have also a rational map
$$
\varphi': (\mathbb P^3)^\vee \dasharrow \overline {\mathcal M_{0,6}}.
$$
It is clear that $\varphi$ is composed of $\varphi'$ and of the obvious rational map
$$
\overline {\mathcal M_{0,6}}\dasharrow \overline {\mathcal M_{2}}
$$
that has degree 6!. 

\subsection{Group action and dominance} Given the curve $C\subset \mathbb Q$, there is an obvious group $G$ of degree $72$ of projectivities of $\mathbb P^3$  fixing $ C$ (and therefore also  $\mathbb Q$). This is an extension of degree 2 of the group of degree $36=3!\cdot 3!$ of automorphisms of $\mathbb P^1\times \mathbb P^1$ that are the products of projectivities of $\mathbb P^1$ that permute three fixed points on the first factor and three fixed points on the second factor.

\begin{lem}\label{lem:gract} Let $\pi_1,\pi_2$ be two tangent planes to $\mathbb Q$ not containing any of the components of $C$. Then $\varphi'(\pi_1)=\varphi'(\pi_2)$ if and only if there is a $g\in G$ such that $g\cdot \pi_1=\pi_2$. 
\end{lem}

\begin{proof} One implication is trivial, so let us prove only the other one. Suppose then that 
$\varphi(\pi_1)=\varphi(\pi_2)$. For $1\leq i\leq 2$, the intersection of $\pi_i$ with $\mathbb Q$ consists of two distinct  lines $a_i, b_i$, and the intersection of $\pi_i$ with $C$ consists of the union of three distinct points $p_{i,j}$ on $a_i$ and three distinct points $q_{i,j}$ on $b_i$, with $1\leq j\leq 3$. By the hypothesis $\varphi'(\pi_1)=\varphi'(\pi_2)$, there is a projectivity $\omega: \pi_1\longrightarrow \pi_2$, that maps 
$a_1\cup b_1$ to $a_2\cup b_2$ and maps the six points 
$p_{1,1},p_{1,2}, p_{1,3}, q_{1,1},q_{1,2}, q_{1,3}$ to the points $p_{2,1},p_{2,2}, p_{2,3}, q_{2,1},q_{2,2}, q_{2,3}$. Let us suppose that $\omega$ maps $a_1$ to $a_2$ and $b_1$ to $b_2$ (otherwise the argument is similar). Then we have unique projectivities 
$\omega_a: a_1\longrightarrow a_2$ and $\omega_b: b_1\longrightarrow b_2$ that map 
$p_{1,1},p_{1,2}, p_{1,3}$ to $p_{2,1},p_{2,2}, p_{2,3}$ and $q_{1,1},q_{1,2}, q_{1,3}$ to 
$q_{2,1},q_{2,2}, q_{2,3}$. Then $\omega_a\times \omega_b: \mathbb P^1\times \mathbb P^1\longrightarrow \mathbb P^1\times \mathbb P^1$ is an element $g$ of $G$ that  maps 
$\pi_1$ to $\pi_2$. 
\end{proof}

\begin{cor}\label{cor:fin} The map $\varphi'$ is generically finite, hence dominant. Thus the same holds for $\varphi$ and
$$
\deg(\varphi)=6! \cdot \deg(\varphi').
$$
\end{cor}

\begin{proof} By Lemma \ref {lem:gract}, there are points in $(\mathbb P^3)^\vee$ (namely the general tangent planes to $\mathbb Q$), whose fibre via $\varphi'$ is finite. The assertion follows. \end{proof}

\subsection{The degree} 

\begin{prop}\label{prop:degfi} The degree of  $\varphi'$ is 144. Hence the degree of $\varphi$ is $144\cdot 6!$. 
\end{prop}

\begin{proof} Let $\pi$ be general a tangent plane to $\mathbb Q$. Then the stabilizer of $\pi$ for the $G$ action does not depend on $\pi$. We claim that this stabilizer is trivial. Indeed, if $g\in G$ fixes $\pi$, then $g$ (that acts also on $(\mathbb P^3)^\vee$)  fixes the dual quadric $\mathbb Q^\vee$ pointwise and so it is the identity. 

So the set theoretic fibre of $\pi$ via $\varphi'$, that by Lemma \ref {lem:gract} is the $G$--orbit of $\pi$, has degree 72. To prove the assertion we will prove that any point in the set theoretic fibre of $\pi$ via $\varphi'$ has multiplicity 2. To show this, given the $G$--action, it suffices to show that $\varphi'$ is simply ramified at $\pi$. To see this we proceed as follows.

Let the intersection curve of $\pi$ with $\mathbb Q$ be the union of the two distinct lines $a,b$, and the intersection of $\pi$ with $C$ be the six points $p_1,p_2,p_3$ lying on $a$ and 
$q_1,q_2,q_3$ lying on $b$. Recall that $\overline {\mathcal M_{0,6}}$ is smooth. We want to identify the tangent space $\mathbb T$ to $\overline {\mathcal M_{0,6}}$ at the pointed curve $(a\cup b, p_1,p_2,p_3,q_1,q_2,q_3)$. 

First we make the following remark. Let $(\mathbb P^1, x_1,\ldots ,x_6)$ be a general point in 
$\overline {\mathcal M_{0,6}}$. Then moving this point in a neighborood of it in $\overline {\mathcal M_{0,6}}$ is the same as keeping $x_1,x_2, x_2$ fixed and moving $x_4,x_5,x_6$. This implies that the tangent space to $\overline {\mathcal M_{0,6}}$ at $(\mathbb P^1, x_1,\ldots ,x_6)$ can be identified with 
$$
\bigoplus_{i=4}^6 T_{\mathbb P^1, x_i}. 
$$
Now, let us take a 1--dimensional flat family $\mathcal C\longrightarrow \mathbb D$ of curves over a disk $\mathbb D$ such that the central fibre is the pointed curve $(a\cup b, p_1,p_2,p_3,q_1,q_2,q_3)$ and the general fibre is the curve $(\mathbb P^1, x_1,\ldots ,x_6)$. By specialization we see that we may identify $\mathbb T$ with
$$
\mathbb T= \bigoplus_{i=1}^3 T_{b, q_i}. 
$$
Now we look at the differential
$$
(d\varphi')_\pi: T_{(\mathbb P^3)^\vee, \pi} \longrightarrow \mathbb T=\bigoplus_{i=1}^3 T_{b, q_i}.
$$
We notice that $(d\varphi')_\pi$ is not injective. Indeed, the tangent vector to the deformation of $\pi$ in the pencil of planes with centre $b$, clearly maps to 0. We want to show that the rank of $(d\varphi')_\pi$ is exactly 2. For this it suffices to prove that  the composition of $(d\varphi')_\pi$ with the projection 
$$
\mathbb T=\bigoplus_{i=1}^3 T_{b, q_i}\longrightarrow \bigoplus_{i=1}^2 T_{b, q_i}
$$
is surjective. To see this, take the tangent vector to a general deformation of $\pi$ that takes $q_1$ fixed. Then the image of this vector to $\bigoplus_{i=1}^2 T_{b, q_i}$ vanishes on the first component but not on the other. Similarly, the image of the tangent vector to a general deformation of $\pi$ that takes $q_2$ fixed vanishes on the second component but not on the other. This proved the assertion. \end{proof} 

\section{The main modular map and its degree}\label{sec:mail}

\subsection{Definition} Let now $X\subset \mathbb P^4$ be a general cubic threefold. 
  If $p\in X$ is a general point, there are six distinct lines on $X_0$ passing through $p$ and 
  these six lines lie on an irreducible quadric cone. This can be easily directly checked, but it follows also by semicontinuity by the results in \S \ref {sec:phi}. Hence it follows again that there is an obvious  rational map
$$
\phi: X \dasharrow  \overline {{\mathcal M}_2}.
$$
By taking into account that the map $\phi_0$ in \S \ref {sec:phi} is dominant, we see that also $\phi$ is dominant. The main purpose of this section is to compute the degree of $\phi$.

\subsection{The degeneration}\label{ssec: deg} By moving $X$ to a Segre primal $X_0$ in a general pencil, we  get a flat family $\mathcal X\longrightarrow \mathbb D$, with $\mathbb D$ a disk, with central fibre the Segre primal $X_0$, with general fibre the general cubic 3--fold $X$ and smooth total space. For reasons that will be clear later, we make an order 2 base change in the family $\mathcal X\longrightarrow \mathbb D$, thus getting a new family $\mathcal X'\longrightarrow \mathbb D$, with central fibre the Segre primal $X_0$, with general fibre the general cubic 3--fold $X$ but total space $\mathcal X'$ singular at the $10$ points that coincide with the double points of the central fibre $X_0$, where $\mathcal X'$ has ordinary double points. 
At this points we desingularize $\mathcal X'$ by blowing it up at its 10 ordinary double points, thus getting a new flat family $\mathcal X''\longrightarrow \mathbb D$ whose general fibre is again the general cubic 3--fold $X$, whereas the central fibre $X''_0$ consists of 11 irreducible components as follows:\\
\begin{inparaenum}
\item [$\bullet$] there is an irreducible component $X'_0$ that is the minimal resolution of the singularities $X'_0\longrightarrow X_0$ of the Segre primal $X_0$ obtained by  blowing it up at the 10 double points. The double points are substituted by the exceptional divisors of the blow--up that are 10 smooth 2--dimensional quadrics;\\
\item [$\bullet$] the remaining 10 components are the exceptional divisors of the blow--up of $\mathcal X'$ at its 10 ordinary double points, i.e., they are 10 smooth quadrics of dimension 3, that are attached to $X'_0$ along the hyperplane sections consisting of the exceptional divisors of the blow--up of $X_0$ at the 10 double points. 
\end{inparaenum}

The map $\phi$ specializes to a map $\phi'': X_0''\dasharrow \overline {\mathcal M_2}$. The restriction of $\phi''$ to $X'_0$ is the composition of $\phi_0$ with the blowing--up map. So its degree is $6!$ by Proposition \ref {prop:deg}. We  have to understand
the restriction of $\phi''$ to the 10 exceptional quadric components. Eventually we will have
\begin{equation}\label{eq:deg}
\deg(\phi)=6! + 10 \cdot \delta
\end{equation}
where $\delta$ is the degree of the restriction of $\phi''$ to one of the 10 exceptional quadric components (remember that the double points of $X_0$ are an orbit by the action of $\mathfrak S_6={\rm Aut}(X_0)$, so that the restriction of $\phi''$ to any one of the 10 exceptional quadric components has the same degree).

\subsection{The restriction of the map to the exceptional components and the degree}\label{ssec:var}  We let $Q$ be one of the 10 exceptional quadric components of the central fibre of $\mathcal X''\longrightarrow \mathbb D$. As we know, $Q$ is attached to the desingularization $X_0'$ of the Segre primal, along a smooth quadric surface $Q'$, that is the exceptional divisor of $X'_0\longrightarrow X_0$ at one of the double points of $X_0$ that we denote by $p$. We will denote by $\mathbb P$ the 3--dimensional span of $Q'$. 

On the quadric $Q'$ there is a curve $C$ of type $(3,3)$ consisting of three lines of a ruling plus three lines of the other ruling. These lines are the exceptional divisors of the blow--ups at $p$ of the six planes in $X_0$  passing through $p$. 

Let now $x\in Q$ be a general point. The lines in $Q$ passing through $x$ form a quadric cone surface $Q_x$ of rank 3, with vertex at $x$, that is the intersection of $Q$ with the tangent hyperplane $T_{Q,x}$ of $Q$ at $x$. This tangent hyperplane intersects:\\
\begin{inparaenum}
\item [$\bullet$] $\mathbb P$  along a plane  $\pi_x$, that, by the generality of $x$ is a general plane of $\mathbb P$ and there is only another point $y\in Q$ such that $\pi_y=\pi_x$;\\
\item [$\bullet$] $Q'$ along a conic $\Gamma_x$, intersection of $\pi_x$ with $Q$, that by the generality assumptions is a general plane section of $Q'$;\\	
\item [$\bullet$] $C$ at six distinct points that are also the intersection of $\Gamma_x$ with $C$.
\end{inparaenum}

Through each one $z$ of the six intersection points of $C$ with $T_{Q,x}$, there are two lines: one in $Q_x$ joining $x$ with $z$, the other in $X'_0$ is described as follows. The point $z$ sits in only one $r$ of the irreducible component lines of $C$. The line $r$ is the exceptional divisor of the blow--up $\alpha'$ in $p$ of a plane $\alpha$ lying in $X_0$ and passing through $p$. Hence there is a unique line in $\alpha$ passing through $p$ whose proper transform on $\alpha'$ passes through $z$. The union of these two lines can be interpreted as a line in the central fibre of $\mathcal X''\longrightarrow \mathbb D$ passing through $x$, so that there are exactly six of these lines. In conclusion, in the above set up we have:

\begin{prop}\label{prop:conc} The restriction of $\phi''$ to $Q$ maps the general point $x\in Q$ to the genus two curve that is the double cover of the conic $\Gamma_x$ branched along the 
six intersection points of $\Gamma_x$ with $C$.
\end{prop} 

As an obvious  consequence, we have:

\begin{cor}\label{cor:cons} The number $\delta$ introduced in equation \eqref {eq:deg}
equals $2\cdot \deg(\varphi)=6!\cdot 288$. 
\end{cor}

From \eqref {eq:deg} we can finally conclude:

\begin{thm}\label{thm:main} One has
$$
\deg(\phi)=6! + 6!\cdot 2880=2074320.
$$
\end{thm}

\end{document}